\documentclass[10pt]{amsart}
\usepackage{amsmath}
\usepackage{amssymb}
\usepackage{amsfonts}
\usepackage{stmaryrd} 
\usepackage{bm} 
\usepackage[pic]{xy}   

\newtheorem{thm}{Theorem}
\newtheorem{cor}{Corollary}
\newtheorem{lem}{Lemma}
\newtheorem{prop}{Proposition}

\newcommand{\x}{\chi}

\newcommand{\om}{\omega}

\renewcommand{\a} {\alpha}

\newcommand{\vf}{\varphi}

\renewcommand\L{\mathrm{L\hspace{.5pt}}}
\renewcommand\S{\mathrm{S\hspace{.1pt}}}
\newcommand\U{\mathrm{U}}
\newcommand\A{\mathrm{A}}

\newcommand\M{\mathrm{M}}
\newcommand{\sd}{\leftthreetimes}

\newcommand{\FF}{\mathbb{F}}
\newcommand\diag{\mathrm{diag}}

\renewcommand\Im{\mathrm{Im\hspace{2pt}}}

\newcommand{\be} {\begin{equation}}
\newcommand{\ee} {\end{equation}}
\newcommand{\ba}[1] {\begin{array}{#1}}
\newcommand{\ea} {\end{array}}
\newcommand{\la} {\langle}
\newcommand{\ra} {\rangle}

\begin{document}
\renewcommand{\refname}{References}
\thispagestyle{empty}

\title{A solvable group isospectral to $\S_4(3)$}
\author{{Andrei V. Zavarnitsine}}%
\address{Andrei V. Zavarnitsine
\newline\hphantom{iii} Group Theory Lab.
\newline\hphantom{iii} Sobolev Institute of Mathematics,
\newline\hphantom{iii} Acad. Koptyug avenue, 4,
\newline\hphantom{iii} 630090, Novosibirsk, Russia
\newline\hphantom{iii} \sc{and}
\newline\hphantom{iii} Mechanics and Mathematics Dept.
\newline\hphantom{iii} Novosibirsk State University,
\newline\hphantom{iii} Pirogova St., 2,
\newline\hphantom{iii} 630090, Novosibirsk, Russia}

\email{zav@math.nsc.ru}%

\thanks{\rm Supported by RFBR, grants 06-01-39001 and 08-01-00322;
by the Council of the President (project NSh-344.2008.1); by the Russian Science Support Foundation (grant of
the year 2008); by SB RAS, Integration Project 2006.1.2; by ADTP ``Development of the Scientific
Potential of Higher School'' of the Russian Federal Agency for Education (Grant 2.1.1.419). }

\maketitle
{\small
\begin{quote}
\noindent{\sc Abstract. }
We construct a solvable group $G$ of order 5\,648\,590\,729\,620 such that the set of element orders of $G$ coincides
with that of the simple group $\S_4(3)$. This completes the determination of finite simple groups isospectral
to solvable groups.
%
%
%
\end{quote}
}

\section{Introduction}

The set of elements orders of a finite group $G$ is denoted by $\om(G)$ and called the {\em spectrum} of $G$.
The spectrum $\om(G)$ contains all divisors of each of its elements, hence is uniquely determined by the set
$\mu(G)$ of its maximal with respect to divisibility elements. Finite groups $G$ and $H$ are {\em isospectral}
if $\om(G)=\om(H)$.

It was shown in \cite{lm} that if a nonabelian simple group $G$ is isospectral to a solvable group then $G\cong
\L_3(3),\U_3(3),\S_4(3),\A_{10}$. The existence of solvable groups isospectral to $\L_3(3)$ and
$\U_3(3)$ was proved in \cite{m02,al}. In \cite{st}, it was shown that there exists no solvable group
isospectral to $\A_{10}$. Thus, the problem remained open only for $\S_4(3)$.
The main result of this paper is as follows:

\begin{thm}\label{main} There exists a solvable group isospectral to the simple group $\S_4(3)$.
\end{thm}

The solvable group from Theorem \ref{main} is constructed as a $17\!\times\! 17$-matrix  group over the field
$\FF_3$ of three elements. This group has order
$$5\,648\,590\,729\,620 = 2^2\cdot3^{24}\cdot 5$$
and is believed to be the smallest
solvable group in question. Although the proofs given in this paper do not depend on computer calculations, the
first evidence for the existence of this group was obtained with the aid of the computer algebra system GAP
\cite{gap}.

As a consequence of Theorem \ref{main}, we have

\begin{cor}\label{cm} A finite simple nonabelian group $G$ is isospectral to a solvable group if and only if
$G\cong \L_3(3),\U_3(3),\S_4(3)$.
\end{cor}

\section{Preliminaries}

From \cite{atl}, it follows that
\be
\mu(\S_4(3))=\{5,9,12\}.
\ee

We will require the following auxiliary result:
\begin{lem}\label{x3}
Let
$$
X=
\left(
  \begin{array}{ccccc}
    1 & x_1 & y_1 & z_1 & t_1 \\
    . & 1 & x_2 & y_2 & z_2 \\
    . & . & 1 & x_3 & y_3 \\
    . & . & . & 1 & x_4 \\
    . & . & . & . & 1 \\
  \end{array}
\right)
$$
be an upper unitriangular $5\!\times\!5$-matrix over a (not necessarily commutative) ring
of characteristic $3$, where the dots stand for zeros. Then $X^9=1$ and $|X|<9$ if and only if
\be\label{9e}
\begin{array}{lr}
x_1x_2x_3=0,&x_2x_3x_4=0,\\[3pt]
\multicolumn{2}{c}{x_1x_2y_3+x_1y_2x_4+y_1x_3x_4=0.}
\end{array}
\ee
\end{lem}
\begin{proof} A direct calculation shows that
$$
X^3=
\left(
  \begin{array}{ccccc}
    1 & . & . & z_1' & t_1' \\
    . & 1 & . & . & z_2' \\
    . & . & 1 & . & . \\
    . & . & . & 1 & . \\
    . & . & . & . & 1 \\
  \end{array}
\right),
$$
where
$$
\begin{array}{lr}
z_1'=x_1x_2x_3,& z_2'=x_2x_3x_4,\\[3pt]
\multicolumn{2}{c}{t_1'=x_1x_2y_3+x_1y_2x_4+y_1x_3x_4.}
\end{array}
$$
The claim follows.
\end{proof}

Let $G$ be a group, $K$ a field, and let $U_1,U_2,U_3$ be right $KG$-modules. Recall that a  $K$-bilinear
map $\vf:U_1\times U_2\to U_3$ is  called {\em balanced \/} if
$$\vf(u_1g,u_2g)=\vf(u_1,u_2)g$$
for all $u_1\in U_1$, $u_2\in U_2$, $g\in G$. It is a well-known universal property of the tensor
product of $KG$-modules that every balanced map factors through '$\otimes$', i.\,e. there exists a unique module
homomorphism $\tilde\vf: U_1\otimes U_2\to U_3$ such that the following diagram commutes
\be\label{diag}
\xymatrix{
&U_1\otimes U_2\ar@{.>}[dr]^{ \tilde{\vf} } &         \\
 U_1\times U_2\ar[rr]^{\vf}\ar[ur]^{\otimes}&         & U_3 }
\ee
Recall also that the exterior square $\wedge^2 U$ of a $KG$-module $U$ is the quotient of $U\otimes U$ by the
submodule generated by all elements of the form $u\otimes u$, $u\in U$.

Henceforth, all the introduced notation will be fixed throughout.

We define the following $4\!\times\!4$-matrices over $\FF_3$:
\be \label{rep}
a=\left(
        \begin{array}{rrrr}

.&1&.&.      \\
.&.&1&.      \\
.&.&.&1      \\
-1&-1&-1&-1  \\
        \end{array}
       \right),
   \qquad
b=\left(
        \begin{array}{rrrr}
 1& .& .& . \\
 .& .& 1& . \\
-1&-1&-1&-1 \\
 .& 1& .& . \\
        \end{array}
       \right),
\ee
Observe that
\be\label{az}
1+a+a^2+a^3+a^4=0.
\ee

Let $F=\la a, b \ra$.  Then $F$ is a Frobenius group of shape $5\!:\!4$ which is isomorphic to the abstract
group
\be
\la a,b\mid a^5=b^4=1, a^b=a^2 \ra,
\ee
and (\ref{rep}) is the matrix representation of
$F$ that corresponds to the simple right $\FF_3F$-module
\be\label{v}
V=\la v\mid vb=v,\ v(1+a+a^2+a^3+a^4)=0\ra
\ee
written in the basis $(v,va,va^2,va^3)$. Observe that $\mu(V\sd F)=\{5,12\}$. In particular,
$a$ acts on $V$ fixed-point freely. Moreover,
\be\label{cb2}
C_V(b^2)=\la v, va^2+va^3\ra.
\ee

The following identities hold in $F$:
\be\label{rels}
\begin{array}{lll}
b^a=ba^4=a^2b, &(b^2)^a=b^2a^2=a^3b^2, & (b^3)^a=b^3a^3=ab^3, \\[3pt]
b^{a^2}=ba^3=a^4b, & (b^2)^{a^2}=b^2a^4=ab^2, &  (b^3)^{a^2}=b^3a=a^2b^3, \\[3pt]
b^{a^3}=ba^2=ab, &  (b^2)^{a^3}=b^2a=a^4b^2, & (b^3)^{a^3}=b^3a^4=a^3b^3, \\[3pt]
b^{a^4}=ba=a^3b, & (b^2)^{a^4}=b^2a^3=a^2b^2, & (b^3)^{a^4}=b^3a^2=a^4b^3.
\end{array}
\ee

Denote by $M$ the right $\FF_3F$-module $\M_4(\FF_3)$ on which the elements of $F$ act by conjugation, i. e.
$m\circ g=g^{-1}mg$ for $m\in M$, $g\in F$.

\begin{lem}\label{mac} Let $m\in M$. The map $v\mapsto m$ can be extended to an $\FF_3F$-module
homomorphism $V\to M$ if $m\in \la b,b^2,b^3\ra_{\FF_3^{\vphantom{A^A}}}$.
\end{lem}
\begin{proof} Since $V$ is a cyclic module, in view of (\ref{v}), it is
sufficient to show that
$$
m\circ b=m, \qquad m\circ (1+a+a^2+a^3+a^4)=0.
$$
The former holds, because $m$ is a linear combination of powers of $b$, hence commutes with $b$. The latter
need only be checked for $m$ equal to one of $b,b^2,b^3$ due to linearity. By (\ref{rels}), for $m=b$ we have
$$
m\circ (1+a+a^2+a^3+a^4)=b+b^a+b^{a^2}+b^{a^3}+b^{a^4}=b(1+a+a^2+a^3+a^4)=0
$$
due to (\ref{az}). Similarly, the claim follows for $m$ equal to $b^2$ or $b^3$.
\end{proof}

It is not difficult to show that the converse of Lemma \ref{mac} holds as well. However, we will not need this
fact.

\begin{lem}\label{wedge} There is an isomorphism of $\FF_3F$-modules
$$
\wedge^2 V \cong V \oplus U
$$
where $U$ is a $2$-dimensional submodule generated by $v\wedge va+v\wedge va^3+va^2\wedge va^3$.
\end{lem}
\begin{proof}
We will need the following fragment  of the ordinary character table of $F$:
$$
  \begin{array}{c|crrrr}
                             &        1a & 4a & 4b & 2a & 5a \\
     \hline
    \x_1^{\vphantom{A^A}}    &  1 & i &-i &-1 & 1 \\
    \x_2                     &  1 &-i & i &-1 & 1\\
    \x_3                     &    4 & . & . & . & -1
  \end{array}
$$
The module $V$ corresponds to a representation of $F$ whose character is $\x_3$. A direct calculation shows
that $\wedge^2\x_3=\x_3+\x_2+\x_1$. Hence, we need only find a $2$-dimensional submodule of $\wedge^2 V$.
Denote
\begin{align*}
&u_1=v\wedge va+v\wedge va^3+va^2\wedge va^3, \\
&u_2=v\wedge va^2 +va\wedge va^2 +va \wedge va^3.
\end{align*}
Then we have
$$
u_1a=u_1,\quad u_1b=u_2,\quad u_2a=u_2,\quad u_2b=-u_1.
$$
Therefore, the $2$-dimensional $\FF_3$-subspace $\la u_1,u_2 \ra$ is a submodule which is generated by
each of the two vectors $u_1$ and $u_2$.

\end{proof}

\section{The group}

We introduce four $17\!\times\!17$-matrices written in the block form as follows:
\begin{align*} \label{gens}
&A=\diag(1,a,a,a,a),\qquad B=\diag(1,b,b,b,b), \\[5pt]
&C=\left(\begin{array}{c|cccc}
1&.&.&.&.      \\
\hline
.&1^{\vphantom{A^A}}&c_1&c_3&.      \\
.&.&1&.&c_4      \\
.&.&.&1&c_2  \\
.&.&.&.&1  \\
        \end{array}
       \right),
   \quad
D=\left(
        \begin{array}{c|cccc}
1 &d& .& .& . \\
\hline
. &1^{\vphantom{A^A}}& .& .& . \\
. &.& 1& .&. \\
. &.& .& 1& . \\
. &.& .& .& 1 \\
        \end{array}
       \right),
\end{align*}
where $d=(1,0,0,0)\in \FF_3^4$, and
\be\label{cdef}
c_1=b, \quad c_2=b^3, \quad c_3=b^2, \quad c_4=-b^2.
\ee

\begin{prop}\label{mp} The group $G=\la A,B,C,D \ra$ is a solvable group of order $2^2\cdot 3^{24}\cdot 5 $
isospectral to $\S_4(3)$.
\end{prop}
\begin{proof}
Clearly, $G\cong P\sd F$, where $P=\la C,D \ra^G$ is the largest normal $3$-subgroup of $G$ and $F$ is
the above-defined Frobenius group of shape $5\!:\!4$ which we identify with $\la A,B \ra$.
In order to determine $\om(G)$, we will study more closely the structure of $P$ and the action of $F$ on $P$.

Each element of $P$ has the form
\be\label{form}
\left(
        \begin{array}{c|cccc}
1 &d_1& d_2& d_3& d_4 \\
\hline
. &1^{\vphantom{A^A}}& f_1& f_3 & h \\
. &.& 1& .& f_4 \\
. &.& .& 1& f_2 \\
. &.& .& .& 1 \\
        \end{array}
       \right)
\ee
for some $d_i\in \FF_3^4$, $f_i,h\in \M_4(\FF_3)$, $i=1,\ldots,4$.

First, we observe that $P$ has exponent $9$.
Indeed, the exponent is at most $9$ by Lemma \ref{x3}. \big(In order
to apply Lemma
\ref{x3}, we may naturally embed $P$ into the upper unitriangular $5\!\times\!5$-matrix group over
$\M_4(\FF_3)$.\big) Moreover, Lemma \ref{x3} implies that the element (\ref{form})
has order $9$ if and only if
\be\label{o9}
d_1(f_1f_4+f_3f_2)\ne 0.
\ee
Thus, there is an element of order $9$ of the form $CD_1$, where $D_1\in
\la D\ra^F$. Indeed, such an element has order $9$ if and only if $0\ne
d_1(c_1c_4+c_3c_2)=d_1(b-b^3)$, where $d_1\in
\FF_3^4$ is the block of $D_1$ in the position $(1,2)$. Since $\la D\ra^F$ is isomorphic to $V=\FF_3^4$ as an
$\FF_3F$-module, $d_1$ can be chosen arbitrarily, and we choose it so that it is not annihilated by the nonzero
matrix $b-b^3$.

We now prove that $P$ has an $F$-invariant normal series with factors isomorphic to $V$ as $\FF_3F$-modules.
 This will imply that $G$ has no elements of order $15$.

It can be seen from (\ref{form}) that
$P=\la D \ra^G\sd\la C \ra^F$, where
the group $\la D \ra^G$ consists of the matrices (\ref{form}) with $f_i,h=0$
and is isomorphic to $V^{\oplus 4}$ as an $\FF_3F$-module; whereas the group
$\la C \ra^F$ consists of the matrices (\ref{form}) with $d_i=0$.
Observe that the components $f_i$ and $h$ in
(\ref{form}) cannot be chosen arbitrarily. If we identify an element of $\la C \ra^F$ with the tuple
$(f_1,f_2,f_3,f_4,h)$, we have
\be\label{multc}
\begin{array}{r@{}l}
(f_1,f_2,f_3,f_4,h)\cdot(f_1',f_2',f_3',f_4',h')=&\\[3pt]
(f_1+f_1',\ f_2+f_2',\ f_3+f_3',\ f_4+f_4',\ &h+h'+f_1f_4'+f_2f_3').
\end{array}
\ee
By Lemma \ref{mac} the maps $\a_i:v\mapsto c_i$, $i=1,\ldots,4$ can be extended to $\FF_3F$-module
homomorphisms $V\to M$.
Let $T$ be subgroup of the group $\{(u,m)\mid u\in V,m\in M\}$, with the multiplication
$$
(u,m)\cdot (u',m')=(u+u',m+m'+u\a_1\cdot u'\a_4+u\a_3\cdot u'\a_2)
$$
and componentwise action of $F$, generated by the element $(v,0)$ as an $F$-group;
i.\,e. $T=\big\la(v,0)\big\ra^F$. Then there is an $F$-group isomorphism $\a: T\to \la C \ra^F$
(meaning that $\a$ commutes with the action of $F$)
defined by the map
\be\label{isom}
\a:(v,0)\mapsto (v\a_1,v\a_2,v\a_3,v\a_4,0)=C.
\ee

Observe that $\big[(u,m),(u',m')\big]=(0,\vf(u,u'))$, where
$$
\vf(u,u')=u\a_1\cdot u'\a_4 - u'\a_1\cdot u \a_4 +u\a_3\cdot u'\a_2- u'\a_3\cdot u\a_2
$$
defines a map $\vf: V\times V\to M$ whose image generates a nonzero $\FF_3F$-module $W$ isomorphic to the derived
subgroup of $\la C\ra^F$. Clearly, $\vf$ is a balanced map of $\FF_3F$-modules, hence $W$ is a homomorphic image
of $V\otimes V$ under $\tilde\vf$ as in (\ref{diag}). Moreover, $\tilde\vf(u\otimes u)=\vf(u,u)=0$ for all $u\in V$,
thus $W$ is a homomorphic image of $\wedge^2 V$.
Using (\ref{rels}), we also have
\begin{align*}
\tilde\vf(v\wedge va\,+&\,v\wedge va^3+va^2\wedge va^3)=
\vf(v,va)+\vf(v,va^3)+\vf(va^2,va^3)=\\
&\ \ \ \, b\cdot(-b^2)^a-b^a\cdot(-b^2)+b^2\cdot(b^3)^a-(b^2)^a\cdot b^3\\
&+b\cdot(-b^2)^{a^3}-b^{a^3}\cdot(-b^2)+b^2\cdot(b^3)^{a^3}-(b^2)^{a^3}\cdot b^3\\
&+b^{a^2}\cdot(-b^2)^{a^3}-b^{a^3}\cdot(-b^2)^{a^2}+(b^2)^{a^2}\cdot(b^3)^{a^3}-(b^2)^{a^3}\cdot (b^3)^{a^2}=\\
&\hspace{100pt}-b^3a^2+b^3a+ba^3-ba\\
&\hspace{100pt}-b^3a+b^3a^3+ba^4-ba^3\\
&\hspace{100pt}-b^3a^3+b^3a^2+ba-ba^4=0.
\end{align*}
Consequently, $\tilde\vf(U)=0$, where $U$ is the $2$-dimensional submodule  of $\wedge^2V$ from Lemma \ref{wedge},
and we thus have $W\cong V$. This proves in particular that $\la C \ra^F$ has nilpotency
class $2$ and order $3^8$. Therefore, $P$ has order $3^{16}\cdot 3^8$ and $A$ acts
fixed-point freely on $P$.

It remains to prove that $G$ has no elements of order $18$. By the Schur--Zassenhaus theorem,
all involutions of $G$ of are conjugate to $B^2$. We will show that $C_P(B^2)$ contains no elements of order $9$.
Every element of $C_P(B^2)$ has the form (\ref{form}) with $d_i\in C_V(b^2)$ and $f_i,h\in C_M(b^2)$.
Moreover, due to the $F$-group isomorphism (\ref{isom}) the components $f_i$ of such an element
must be of the form $u\a_i$ for some $u\in C_V(b^2)$. Thus, in view of the condition (\ref{o9}),
we need only show that
\be\label{lk}
w\psi(u)=0
\ee
for all $w,u\in C_V(b^2)$, where $\psi: C_V(b^2)\to M$ is given by
$$
\psi(u)=u\a_1\cdot u\a_4+u\a_3\cdot u\a_2.
$$
Let $N$ be the $\FF_3\la b\ra$-module generated by $\Im\psi$. Observe that $C_V(b^2)$ is $\la b\ra$-invariant.
Hence, the condition (\ref{lk}) has to be checked only when $\psi(u)$ is a generator of $N$.
Using (\ref{cb2}) it can be shown that $N$ is generated as a module by $\psi(v)$ and $\psi(va^2+va^3)$.
\big(Indeed, the map $\psi$ is quadratic in $u$, hence $N$ is a homomorphic image of the symmetric square
of  $C_V(b^2)$ which can be generated by two elements.\big) By (\ref{rels}) and (\ref{az}), we have
$$
\psi(v)=b\cdot(-b^2)+b^2\cdot b^3=(1-b^2)b,
$$
and
\begin{align*}
\psi(va^2&+va^3)=\\
&(b^{a^2}+b^{a^3})((-b^2)^{a^2}+(-b^2)^{a^3})+((b^2)^{a^2}+(b^2)^{a^3})((b^3)^{a^2}+(b^3)^{a^3})=\\
-&(ba^3+ba^2)(b^2a^4+b^2a)+(b^2a^4+b^2a)(b^3a+b^3a^4)=\\
-&b^3(a+a^2+a^3+a^4)+b(a+a^2+a^3+a^4)=-(1-b^2)b.
\end{align*}
Thus, in both cases, $\psi(u)$ is divisible from the left by $1-b^2$. However, if $w\in C_V(b^2)$,
we have $w(1-b^2)=0$; therefore, (\ref{lk}) holds as is required.
\end{proof}

Theorem \ref{main} is now a consequence of Proposition \ref{mp}.

{\em Acknowledgement.} The author is thankful to Prof. V. D. Mazurov for discussing the
problem and for the remarks about this paper.

\end{document}